\documentclass[12pt]{amsart}
\usepackage{amssymb,latexsym}
\usepackage{enumerate}

\makeatletter
\@namedef{subjclassname@2010}{%
  \textup{2010} Mathematics Subject Classification}
\makeatother

\newtheorem{thm}{Theorem}[section]
\newtheorem{cor}[thm]{Corollary}
\newtheorem{lemma}[thm]{Lemma}

\theoremstyle{definition}
\newtheorem{defin}[thm]{Definition}
\newtheorem{rem}[thm]{Remark}

\numberwithin{equation}{section}

\def\C{\mathbb{C}}
\def\R{\mathbb{R}}
\def\Z{\mathbb{Z}}
\def\N{\mathbb{N}}

\def\H{{\mathcal{H}}}
\def\s{\textbf{{s}}}
\def\t{\textbf{{t}}}

\begin{document}

\baselineskip=17pt

\title[Atomic decomposition and complex interpolation]{Atomic decomposition and interpolation via the complex method for mixed norm Bergman spaces on tube domains over symmetric cones}
\author{DAVID BEKOLLE, JOCELYN GONESSA AND CYRILLE NANA}
%\address{Current address of Bekolle: University of Yaound\'e 1, Faculty of Sciences, Department of Mathematics, P. O. Box 812, Yaound\'e-Cameroon.}
\address{University of Ngaound\'er\'e, Faculty of Science, Department of Mathematics, P. O. Box 454, Ngaound\'er\'e, Cameroon}
\email{dbekolle@univ-ndere.cm}
\address{Universit\'e de Bangui, Facult\'e des Sciences, D\'epartement de Math\'ema-tiques et Informatique, BP. 908, Bangui, R\'epublique Centrafricaine}
\email{gonessa.jocelyn@gmail.com}
\date{}
\address{Faculty of Science, Department of Mathematics, University of Buea, P.O. Box 63, Buea, Cameroon}
\email{{\tt nana.cyrille@ubuea.cm}}

\maketitle
  
%%%%%%%%%%%%%%%beginig  %%%%%%%%%%%%%%%%%%%%%%%%%%%%%%%%%%%%
\begin{abstract}
Starting from an adapted Whitney decomposition of tube domains in $\C^n$ over irreducible symmetric cones of $\R^n,$ we prove an atomic decomposition theorem in mixed norm weighted Bergman spaces on these domains. We also characterize the  interpolation space via the complex method between two mixed norm weighted Bergman spaces.
\end{abstract}

\section{Introduction}
%Let $\Omega$ be an irreducible cone in $\mathbb R^n.$ We denote by $(\cdot| \cdot)$ the canonical scalar product in $\mathbb R^n$ with respect to which the cone $\Omega$ is self-dual. The tube domain over the cone $\Omega$ in the complexification of $\mathbb R^n$ is defined by $T_\Omega = \mathbb R^n.$ As in the text \cite{FK}, we respectively denote $r$ and $\Delta_j (x)$ the rank of the cone $\Omega$ and the $j$-principal minor on 

The context and the notations are those of \cite{FK}.
Let $\Omega$ be an irreducible symmetric cone of rank r in a vector
space $V$ of dimension n, endowed with an inner product $(.|.)$
for which $\Omega$ is self-dual.
%It is well-known that
%$\Omega$ induces in $V\sim \bR^n$ a structure of Euclidean
%Jordan algebra, in which $\overline {\Omega}=\{x^2:x\in V\}$. We
%denote by $e$ the identity element in $V$ and by $(x/y)=tr(xy)$
%the canonical inner product.

%Let $G(\Omega)$ be the group of transformations of $\Omega$, and
%$G$ its identity component. It is well-known that there exists a
%subgroup $H$ of $G$ acting simply transitively on $\Omega$, that
%is every $y\in \Omega$ can be written uniquely as $y=g\mathbf e$
%for some $g\in H$ and a fixed $\mathbf e\in \Omega$.

We recall that $\Omega$ induces in $V$ a structure of Euclidean
Jordan algebra with identity $\mathbf e$ such that $$\overline
\Omega=\{x^2: x\in V\}.$$  Let
$\{c_1,\cdots,c_r\}$ be a fixed Jordan frame in $V$ and
$$V=\underset{1\le i\le j\le r } {\oplus} V_{i,j}$$ be its associated Peirce
decomposition of $V$. We denote by
$$\Delta_1(x),\cdots,\Delta_r(x)$$ the principal minors of $x\in V$
with respect to the fixed Jordan frame $\{c_1,\cdots,c_r\}$. More
precisely, $\Delta_k(x), \hskip 2truemm k=1,\cdots,r$ is the determinant of the projection
$P_kx$ of $x$, in the Jordan subalgebra $V^{(k)}=\oplus_{1\le i\le
j\le k }V_{i,j}$. We have  $\Delta_k(x)>0$,
$k=1,\cdots,r$ when $x\in \Omega,$ and the determinant $\Delta$ of the Jordan algebra  is given by $\Delta=\Delta_r.$  The generalized power function
on $\Omega$ is defined as
$$\Delta_{\bf s} (x)=\Delta_1^{s_1-s_2}(x)\Delta_2^{s_2-s_3}(x)\cdots\Delta_r^{s_r}(x), \quad x\in \Omega, \hskip 2truemm {\s=(s_1,\cdots,s_r)}\in\C^r.$$
We adopt the following standard notations:
$$ n_k = 2(k-1)\frac{\frac nr -1}{r-1} \quad {\rm and}\quad m_k = 2(r-k)\frac{\frac nr -1}{r-1}.$$ 
%For all  $\s=(s_1,\cdots,s_r)$ and $\t=(t_1,\cdots,t_r)$ in $\R^r$, 
%we write
%$$\s\leq\t \quad \rm if \quad s_j\leq t_j \quad {\rm for \hskip 1truemm all} \quad j=1,\cdots,r.$$
%We also denote
%$$q_\s=\min_{1\leq j\leq r}\left(1+\frac{s_j-\frac{m_j}{2}}{\frac {n_j}2}\right);$$
%$$q_\s(p)=\min\{p,p'\}q_\s;$$
%$${q}_{\s,\t}=\min_{1\leq j\leq r}\left(1+\frac{s_j+\frac {m_j}2-t_j}{t_j-\frac {n_j}2}\right). {\bf {to be checked!!!}}$$
For $\s=(s_1,\cdots,s_r) \in \mathbb R^r$ and $\rho$ real, the notation $\s+\rho$ will stand for the vector whose coordinates are $s_k + \rho, \hskip 2truemm k=1,\cdots,r.$
For $1\leq p\leq\infty$ and $1\leq q <\infty,$ let
$L^{p,\,\,q}_\textbf s$ 
denote the mixed norm Lebesgue space constisting of measurable functions $F$ on $T_\Omega$ such that
$$\|F\|_{L^{p,\,\,q}_{\s}}=\left(\int_\Omega\|F(\cdot+iy)\|^q_p\Delta_{\s-\frac{n}{r}}(y)dy\right)^\frac{1}{q}<\infty$$
where  
$$\|F(\cdot+iy)\|_p=\left(\int_{V}|F(x+iy)|^pdx\right)^\frac{1}{p}$$
(with the obvious modification if $p=\infty$). The \textit{mixed norm  weighted Bergman space} $A_\s^{p,q}$ is the (closed) subspace of $L^{p,\,\,q}_\s$ consisting of holomorphic functions. Following \cite{DD}, $A_\s^{p,q}$ is non-trivial if only if $\s_k>\frac {n_k}2, \hskip 2truemm k=1,\cdots, r.$ When $p=q$, we write $L^{p,\,\,q}_\s=L^{p}_\s$ and $A^{p,\,\,q}_\s=A^{p}_\s$ which are respectively the usual weighted Lebesgue space and the usual weighted Bergman space. Moreover, when $p=q=2$ the orthogonal projector $P_\s$ from the Hilbert space $L^2_\s$ onto its closed subspace $A^2_\s$ is called \textit{weighted Bergman projector}. It is well known that $P_\s$ is the integral operator on $L_\s^2$ given by the formula 
$$P_\s F(z)=\int_{T_\Omega}B_\s (z, u+iv)F(u+iv)\Delta_{\s-\frac{n}{r}}(v)dudv,$$
where
$$B_\s (z,u+iv)=d_\s \Delta_{-\s-\frac{n}{r}}(\frac{z-u+iv}{i}) $$
is the reproducing kernel on $A^2_\s,$  called \textit{weighted Bergman kernel} of $T_\Omega$. Precisely,
 $\Delta_{-\s-\frac{n}{r}}(\frac{x+iy }{i})$ is the holomorphic determination of the $(-\s-\frac{n}{r})$-power which reduces to the function $\Delta_{-\s-\frac{n}{r}}(y )$ when $x =0$.

Our first result is an atomic decomposition theorem for functions in mixed norm weighted Bergman spaces on tube domains over symmetric cones. It generalizes the result of \cite{BBGNPR} for usual weighted Bergman spaces on tube domains over symmetric cones and the result of \cite{RT} for mixed norm weighted Bergman spaces on the upper half-plane (the case $n=r=1).$ 

\newtheorem*{thA}{\bf Theorem A}
\begin{thA}{\it
Let  $\textbf s$ be a vector of $\R^r$ such that $s_k > \frac {n_k}2, \hskip 1truemm k=1,\cdots,r.$ Assume that $P_\s$ extends to a bounded operator on $L_\textbf s^{p,q}$. Then there is a sequence  of 
points $\{z_{l,j}=x_{l,j}+iy_j\}_{l\in\Z, \hskip 1truemmj\in\N}$ in $T_\Omega$ and a positive constant $C$ such that the following assertions hold.
\begin{enumerate}[\upshape (i)]
\item For every sequence $\{\lambda_{l,j}\}_{l\in\Z,  \hskip 1truemm j\in\N}$ such that $$\sum_j\left(\sum_l|\lambda_{l,j}|^p\right)^\frac{q}{p}\Delta_{\textbf s+\frac{nq}{rp}}(y_j)<\infty,$$ the series $$\sum_{l,j}\lambda_{l,j}\Delta_{\textbf s+\frac{nq}{rp}}(y_j)B_\s(z,z_{l,j})$$ is convergent in $A_\textbf s^{p,q}$. Moreover, its sum $F$ satisfies the inequality
$$\|F\|_{A_\textbf s^{p,q}}^q\leq C \sum_j\left(\sum_l|\lambda_{l,j}|^p\right)^\frac{q}{p}\Delta_{\textbf s+\frac{nq}{rp}}(y_j)$$
\item Every function $F\in A_\textbf s^{p,q}$ may be written as 
$$F(z)=\sum_{l,j}\lambda_{l,j}\Delta_{\textbf s+\frac{nq}{rp}}(y_j)B_\textbf s (z,z_{l,j}),$$ 
with $$\sum_j\left(\sum_l|\lambda_{l,j}|^p\right)^\frac{q}{p}\Delta_{\textbf s+\frac{nq}{rp}}(y_j)\leq C \|F\|_{A_\textbf s^{p,q}}^q$$
\end{enumerate} 
}
\end{thA}
 
Our second result is an interpolation theorem between mixed norm weighted Bergman spaces. It generalizes the result of \cite{BGN1} for usual weighted Bergman spaces. We adopt the following notation.
$$q_{\s}=\min_{1\leq k\leq r}\left(1+\frac{s_k-\frac{n_k}{2}}{\frac {m_k}2}\right).$$
%$$p_{\s} =1+\min \frac {s_j +\frac nr}{(m_j - s_j)_+}$$
%and
%$$q_\s(p)=\min\{p,p'\}q_\s.$$
%$${q}_{\s,\t}=\min_{1\leq j\leq r}\left(1+\frac{s_j+\frac {m_j}2-t_j}{t_j-\frac {n_j}2}\right). {\bf {to be checked!!!}}$$

\newtheorem*{thB}{\bf Theorem B} 
\begin{thB}{\it
\begin{enumerate}[\upshape (1)]
\item 
Let $\s_0, \s_1\in \mathbb R^r$ be such that $(s_0)_k, (s_1)_k > \frac nr -1,\quad k=1,\cdots,r$ and let $1\leq p_0, p_1\leq \infty$, $1\leq q_0, \hskip 1truemm q_1<\infty$ be such that
%\begin{enumerate}
%\item
$1\leq q_i < q_{{\s}_i}$ for every $i=0, 1.$
%\item
%$\frac 1{q_{\s_i} (p_i)}<\frac 1{q_i}<1-\frac 1{q_{\s_i} (p_i)}, i=0, 1.$ 
%\end{enumerate}
Then
for every $\theta \in (0, 1),$ we have
$$[A_{{\s}_0}^{p_0,q_0},\,\,A_{{\s}_1}^{p_1,q_1}]_\theta=A_{\s}^{p,q}$$ 
%be such that $1q_i < q_{\s_i}$ for every $i=0, 1.$ Then for every $\theta \in (0, 1),$ we have
%$$[A_{\s}^{p_0,q_0},\,\,A_{\s}^{p_1,q_1}]_\theta=A_\s^{p,q}$$ 
with equivalent norms, where $\frac{1}{p}=\frac{1-\theta}{p_0}+\frac{\theta}{p_1}, \frac{1}{q}=\frac{1-\theta}{q_0}+\frac{\theta}{q_1}$ and $\frac{\s}{q}=\frac{(1-\theta){\s}_0}{q_0}+\frac{\theta {\s}_1}{q_1}$. 
%Let $\s\in \mathbb R^r$ be such that $\s>\textbf{d}$ and let $1\leq p_0, p_1\leq \infty$, $1\leq q_0<q_1\leq \infty$ be such that $q_j < q_{\s} (p_j)$ for every $j=0, 1.$ 
%Then for every $\theta\in (0,1)$, we have
%$$[A_{\s}^{p_0,q_0},\,\,A_{\s}^{p_1,q_1}]_\theta=A_\s^{p,q}$$ with equivalent norms, where $\frac{1}{p}=\frac{1-\theta}{p_0}+\frac{\theta}{p_1}$ and $\frac{1}{q}=\frac{1-\theta}{q_0}+\frac{\theta}{q_1}.$
\item 
Let $\s \in \mathbb R^r$ be such that $s_k > \frac {n_k}2, \hskip 2truemm k=1,\cdots, r.$ 
%Let $1\leq p_0, p_1\leq \infty$ be arbitrary if %$s_j > \frac nr -1, \hskip 2truemm j=1,\cdots, r$ and let $1\leq p_0, p_1 < p_{\s}$  otherwise. Let  $1\leq q_0<q_1< \infty$ be such that
%$\frac 1{q_{{\s}_i} (p_i)}<\frac 1{q_i}<1-\frac 1{q_{{\s}_i} (p_i)}, \hskip 2truemm i=0, 1.$ 
Assume that $P_{\s}$ extends to a bounded operator on $L^{p_i, q_i}_{\s}, \hskip 2truemm i=0, 1$ for $1\leq p_0, p_1 \leq \infty$ and $1 <q_0, q_1 < \infty.$ 
Then
for every $\theta \in (0, 1),$ we have
$$[A_{\s}^{p_0,q_0},\,\,A_{\s}^{p_1,q_1}]_\theta=A_{\s}^{p,q}$$ 
with equivalent norms, where $\frac{1}{p}=\frac{1-\theta}{p_0}+\frac{\theta}{p_1}$ and $\frac{1}{q}=\frac{1-\theta}{q_0}+\frac{\theta}{q_1}.$
\item Let ${\s} \in \mathbb R^r$ be such that $(s)_k > \frac nr -1, \hskip 2truemm k=1,\cdots, r,$  let $1\leq p_0 < p_1\leq\infty$ and let $q_0, q_1$ be such that $1\leq q_0 < q_{\s}\leq q_1.$ 
%$\frac{1}{p}=\frac{1-\theta}{p_0}+\frac{\theta}{p_1}$ and $\frac{1}{q}=\frac{1-\theta}{q_0}+\frac{\theta}{q_1}$
We assume that $P_{\s}$  extends to a bounded operator on $L^{p_1, q_1}_{\s}.$
Then for some values of $\theta\in (0,1)$, we have
$$[A_{\s}^{p_0,q_0},\,\,A_{\s}^{p_1,q_1}]_\theta=A_{\s}^{p,q}$$ with equivalent norms, where $\frac{1}{p}=\frac{1-\theta}{p_0}+\frac{\theta}{p_1}$ and $\frac{1}{q}=\frac{1-\theta}{q_0}+\frac{\theta}{q_1}.$  
%{\bf (to be checked!!!)} 
\end{enumerate}
}
\end{thB}

We recall sufficient conditions on $p, q$ and $\s$ under which the weighted Bergman projector $P_{\s}$ extends to a bounded operator on $L^{p, q}_{\s}.$ We adopt the following notations.
$$p_{\s} =1+\min \limits_k \frac {s_k +\frac nr}{(m_k - s_k)_+}$$
and
$$q_\s(p)=\min\{p,p'\}q_\s.$$

\begin{thm}[\cite{DD1} and \cite{NT}]
The weighted Bergman projector $P_{\s}$ extends to a bounded operator on $L^{p, q}_{\s}$ whenever $\frac 1{q_{\s} (p)}<\frac 1{q}<1-\frac 1{q_{\s} (p)}$ in the following two cases:
\begin{enumerate}
\item[\rm (i)]
$s_j > \frac {n_j}2, \hskip 2truemm j=1,\cdots r$ and $1\leq p < p_{\s}$ \cite{DD1}; 
\item[\rm (ii)]
$s_j > \frac nr -1, \hskip 2truemm j=1,\cdots r$ and $1\leq p \leq \infty$ \cite{NT}.
\end{enumerate}
\end{thm}

We restrict to tube domains over Lorentz cones ($r=2$). For real $\s = (s, s),$ this problem was recently completely solved on tube domains over Lorentz cones (cf. \cite{BGN} for a combination of results from \cite{BD} and \cite{BBGR};  cf. also \cite{BN} for the unweighted case  $\s=(\frac nr,\cdots,\frac nr)$). For vectorial $\s = (s_1, s_2),$ Theorem 1.1 was also extended in \cite{BGN} to other values $p$ and $q.$

%\begin{thm}
%We denote by $\Lambda_n$ the Lorentz cone in $\mathbb R^n, \hskip 2truemm n\geq 3.$ The weighted Bergman projector $P_s$ ($s$ real) of $T_{\Lambda_n}$ extends to a bounded operator on $L^{p, q}_s$ if and only if $p, q$ and $s$ satisfy the following conditions.
%\begin{enumerate}
%\item[\rm 1)]
%$p\leq \frac {2n}{n-2}$ and $q< \frac {s + \frac n2 -1}{\frac n{2p'} -1};$
%\item[\rm 2)]
%$2\leq p < \frac {2n}{n-2}$ and $q < \frac {2(s-\frac n2 +1)}{\frac n2 -1}.$
%\end{enumerate}
%\end{thm}

The plan of this paper is as follows. In Section 2, we overview some preliminaries and useful results about symmetric cones and tube domains over symmetric cones. In Section 3, we study atomic decomposition of mixed norm Bergman spaces and we prove a more precise statement of Theorem A. In Section 4, we study interpolation via the complex method between mixed norm weighted Bergman spaces and we prove Theorem B. In particular we give a more precise statement of assertion (3) of this theorem (Theorem 4.6) and we ask an open question. A final remark will point out a connection between the two main theorems of the paper (Theorem A and Theorem B).\\
\indent
For $\s=(s,\cdots, s)$ real, Theorem A and Theorem B were presented in the PhD dissertation of the second author \cite{G}. 

\section{Preliminaries}
Materials of this section are essentially from \cite{FK}. We give some definitions and useful results.\\
Let $\Omega$ be an irreducible symmetric cone of rank $r$ in a real vector space $V$ of dimension $n$ endowed with the structure of Euclidean Jordan algebra with identity $\textbf e$. In particular, $\Omega$ is self-dual with respect to the inner product 
$$(x|y)=\textbf{tr}(xy)$$
on $V$.

\subsection{Group action}
Let $G(\Omega)$ be the group of linear transformations of the cone $\Omega$ and $G$ its identity component. By definition, the subgroup $G$ of $G(\Omega)$ is a semi-simple Lie group which acts transitively on $\Omega.$   This gives the identification $\Omega \sim G/K$, where $K:=\{g\in G:\,\,\,\,g\cdot \mathbf e=\mathbf e\}$ is a maximal compact subgroup of $G$. More precisely,
$$K=G\cap O(V),$$
where $O(V)$ is the orthogonal group in $V.$ Furthermore,  there is a solvable subgroup $T$ of $G$ acting simply transitively on $\Omega$. That is, every $y\in\Omega$ can be written uniquely as $y=t\cdot \mathbf e$, for some $t\in T.$ 
%Also it is well known in \cite{FK} that for symmetric cone $\Omega$, its underlying vector space %$\R^n$ can be endowed with a multiplication rule which makes it a \textit{Euclidean Jordan algebra} %with identity element $e$. With such multiplication, $\overline{\Omega}$ coincides with the set %$\{x^2:\,\,\,\,x\in\R^n\}$. 

%We say that $\Omega$ is irreductible when $\R^n$ cannot be decomposed as a direct sum of two %lower-dimensional subspaces which contain a pair of symmetric cones whose direct some equals %$\Omega$ (alternatively, when $V$ does not contain non-trivial ideals).

%Suppose now that the cone $\Omega$ in $\R^n$ is irreductible, has rank $r$. 
Let  $\{c_1,\cdots,c_r\}$ in $\R^n$ be a fixed Jordan frame in $V$ (that is, a complete system of idempotents) and 
$$V=\underset{1\leq i\leq j\leq r}{\oplus}V_{i,j}$$
be its associated Peirce decomposition of $V$
 where 
$$
\left\{
\begin{array}{ll}
V_{i,i}=\R c_i\\
V_{i,j}=\{x\in V:\,\,c_i x=c_j x=\frac{1}{2}x\}\,\,\textrm{if}\,\,i<j.
\end{array}
\right.
$$
We have $\mathbf e = \sum \limits_{1\leq i\leq r} c_i.$
Then the solvable Lie group $T$  factors as the semidirect product $T=NA=AN$ of a nilpotent subgroup $N$ consisting of lower triangular matrices, and an abelian subgroup $A$ consisting of diagonal matrices. The latter takes the explicit form
$$A=\{P(a):\,\,\,\, a=\sum_{i=1}^r a_ic_i,\,\,\,a_i>0\},$$
where $P$ is the quadratic representation of $\R^n$. This also leads to the Iwasawa and Cartan decompositions of the semisimple Lie group $G:$ $$G=NAK \quad {\rm and} \quad G=KAK.$$
Still following \cite{FK}, we shall denote by $\Delta_1(x),\cdots,\Delta_r(x)$ the principal minors of $x\in V$, with respect to the fixed Jordan frame $\{c_1,\cdots,c_r\}$. These are invariant functions under the group $N$,
$$\Delta_k(nx)=\Delta_k(x),$$
where $n\in N$, $x\in V$, $k=1,\cdots,r$, and satisfy a homogeneity relation under $A$,
$$\Delta_k(P(a)x)=a_1^2\cdots a^2_k\Delta_k(x),$$
if $a=a_1c_1+\cdots+a_rc_r$. 

The determinant function $\Delta(y)=\Delta_r(y)$ is also invariant under $K$, and moreover, satisfies the formula
\begin{equation}\label{Delta}
\Delta(gy)=\Delta(ge)\Delta(y)=Det^\frac{r}{n}(g)\Delta(y),\,\,\forall g\in G,\,\,\forall y\in\Omega
\end{equation}
where $Det$ is the usual determinant of linear mappings.  
It follows from this formula that the measure $\frac{d\xi}{\Delta^{\frac{n}{r}}(\xi)}$ is $G$-invariant in $\Omega.$ 

Finally, we recall the following  version of Sylvester's theorem.
$$\Omega=\{x\in\R^n:\,\,\,\,\Delta_k(x)>0, \hskip 2truemm k=1,\cdots,r\}.$$
\subsection{Geometric properties}
With the identification $\Omega\sim G/K$, the cone can be regarded as a Riemannian manifold with the $G$-invariant metric defined by
$$<\xi,\eta>_y:=(t^{-1}\xi|t^{-1}\eta)$$
if $y=t\cdot \textbf e$ with $t\in T$ and $\xi$ and $\eta$ are tangent vectors at $y\in\Omega$. We shall denote by $d_\Omega$ the corresponding invariant distance, and by $B_\delta(\xi)$ the associated ball centered at $\xi$ with radius $\delta$. Note that for each $g\in G$, the invariance of $d_\Omega$ implies that $B_\delta(g\xi)=gB_\delta(\xi)$. We also note that 
\begin{itemize}
\item
on  compact sets of $\mathbb R^n$ contained in $\Omega,$ the invariant distance $d_\Omega$ is equivalent to the Euclidean distance in $\mathbb R^n;$
\item
the associated balls $B_\delta$ in $\Omega$ are relatively compact in $\Omega$.
\end{itemize}
We also need the following crucial invariance properties of $d_\Omega$ and $\Delta_k$, obtained in \cite{BBGNPR, BBGR}.
\begin{lemma}
Let $\delta_0>0$. Then there is a constant $\gamma0$ depending only on $\delta_0$ and $\Omega$ such that for every $0<\delta\leq\delta_0$ and 
for $\xi$, $\xi'\in\Omega$ satisfying $d_\Omega(\xi,\xi')\leq \delta$ we have $$\frac{1}{\gamma}\leq\frac{\Delta_k(\xi)}{\Delta_k(\xi')}\leq\gamma,\,\,\forall k=1,\cdots,r.$$ 
\end{lemma}
\begin{lemma}
Let $\delta_0>0$ be fixed. Then there exist two constants $\eta_1>\eta_2>0$, depending only on $\delta_0$ and $\Omega$, such that for every $0<\delta\leq\delta_0$ we have
$$\{|\xi-e|<\eta_2\delta\}\subset B_\delta(e)\subset\{|\xi-e|<\eta_1\delta\}$$ 
\end{lemma}

The next corollary is an easy consequence of the previous lemma.

\begin{lemma}\label{lemma5}
Let $\delta_0=1$. Then there is a positive constant $\gamma$  such that for every $\delta \in (0, 1)$ such that $\eta_1 \delta < 1,$ we have
$$B_\delta(\xi)\subset\{y\in\Omega:\,\,\,\,y-\gamma\xi\in\Omega\}$$
for all $\xi\in\Omega$.
\end{lemma}

\subsection{Gamma function in $\Omega$} The generalized gamma function in $\Omega$ is defined in terms of  the generalized power functions  by
$$\Gamma_\Omega(\s)=\int_\Omega e^{-(\xi|e)}\Delta_{\textbf s}(\xi)\frac{d\xi}{\Delta^\frac{n}{r}(\xi)} \quad \quad (\s =(s_1,\cdots,s_r)\in \mathbb R^r).$$
This integral is known  to converge absolutely if and only if 
$\Re e  \hskip 1truemm  s_k >\frac {n_k}2, \hskip 2truemm k=1,\cdots, r.$ In this case,
$$\Gamma_\Omega(\s)=(2\pi)^\frac{n-r}{2}\prod_{k=1}^r\Gamma \left (s_k-\frac{n_k}2 \right ),$$
where $\Gamma$ is the classical gamma function.  We shall denote $\Gamma_\Omega(\s)=\Gamma_\Omega (s)$ when $\s=(s,\cdots,s)$. In view of \cite{FK},  the Laplace transform of a generalized power function is given for all $y\in\Omega$ by
$$\int_\Omega e^{-(\xi|y)}\Delta_{\textbf s}(\xi)\frac{d\xi}{\Delta^\frac{n}{r}(\xi)}=\Gamma_\Omega(\s)\Delta_\s(y^{-1})$$
for each $\s\in\C^r$ such that $\Re e  \hskip 1truemm  s_k >\frac {n_k}2$ for all $k=1,\cdots,r$. We recall that %$y=t\cdot \textbf e$ if and only if $y^{-1}=t^{*-1}\cdot \textbf e$ with $t\in T$.
$y^{-1}=t^{*-1}\cdot \textbf e$ whenever $y=t\cdot \textbf e$ with $t\in T.$ Here $t^\star$ denotes the adjoint of the transformation  $t\in T$ with respect to the inner product $(\cdot|\cdot).$
\\
The power function $\Delta_{\textbf s}(y^{-1})$ can be expressed in terms of the rotated Jordan frame $\{c_r,\cdots,c_1\}$. Indeed if we denote by $\Delta^*_k$, $k=1,\cdots,r$, the principal minors with respect to the rotated Jordan frame $\{c_r,\cdots,c_1\}$ then
$$\Delta_{\textbf s}(y^{-1})=[\Delta^*_{\s^*}(y)]^{-1},\,\,\forall \s=(s_1,\cdots,s_r)\in\C^r.$$
Here $\s^*=(s_r,\cdots,s_1)$.

\subsection{Bergman distance on the tube domain $T_\Omega$} Following \cite{BBGNPR}, we define a matrix function $\{g_{j,k}\}_{1\leq j,k\leq n}$ on $T_\Omega$ by
$$g_{j,k}(z)=\frac{\partial^2}{\partial z_j\partial \bar z_k}\log B(z,z)$$
where $B$ is the unweighted Bergman kernel of $T_\Omega,$ i.e $B=B_\s$ with $\s =(\frac nr,\cdots \frac nr)$. The map 
$$T_\Omega\ni z\mapsto \H_z$$
with $$\H_z(u,v)=\sum_{j,k=1}^n g_{j,k}(z)u_k\bar v_k,\,\,\,u=(u_1,\cdots,u_n),\,\,v=(v_1,\cdots,v_n)\in\C^n$$
defines a Hermitian metric on $\C^n$, called the Bergman metric. The Bergman length of a smooth path $\gamma:[0,1]\to T_\Omega$ is given by
$$l(\gamma)=\int_0^1\{\H_{\gamma(t)}(\overset{.}{\gamma}(t),\overset{.}{\gamma}(t))\}^\frac{1}{2}dt$$
and the Bergman distance $d(z_1,z_2)$ between two points $z_1$, $z_2$ of $T_\Omega$ is
$$d(z_1,z_2)=\inf_{\gamma}l(\gamma)$$
where the infimum is taken over all smooth paths $\gamma:[0,1]\to T_\Omega$ such that $\gamma(0)=z_0$ and $\gamma(1)=z_1$. It is well known that the Bergman distance $d$ is equivalent to the Euclidean distance on the compact sets of $\C^n$ contained in $T_\Omega$ and the Bergman balls in $T_\Omega$ are relatively compact in $T_\Omega$. Next, we again denote by $\R^n$  the group of translations by vectors in $\R^n$. Then the group $\R^n\times T$ acts simply transitively on $T_\Omega$ and  the Bergman distance $d$ is invariant under the automorphisms of $\R^n\times T$. 

\subsection{A Whitney decomposition of the tube domain $T_\Omega$}
In the sequel, the Bergman ball in $T_\Omega$ with centre at $z$ and radius $\eta$ will be denoted  $\mathbf{B}_\eta(z).$
\begin{lemma}\label{inclusion}
There exists a constant $R>1$ such that for all $\eta\in (0,1)$ and $z_0=x_0+iy_0\in T_\Omega$, the following inclusions hold :
$$\{x+iy\in T_\Omega:\,\,\|g^{-1}(x-x_0)\|<\frac{\eta}{R} \hskip 2truemm \textrm{and}\,\,y\in B_\frac{\eta}{R}(y_0)\}\subset\mathbf{B}_\eta(z_0), $$
$$ \mathbf{B}_\eta(z_0)\subset \{x+iy\in T_\Omega:\,\,\|g^{-1}(x-x_0)\|< {R\eta}\,\,\textrm{and}\,\,y\in B_ {R\eta} (y_0)\}$$
where $g$ is the element of $T$ satisfying $g\cdot e=y_0$.
\end{lemma}
\begin{proof}
From the invariance under translations and automorphisms of $T$ we have that
$$g^{-1}(x-x_0)+ig^{-1}y\in {\mathbf{B}_\eta(ie)}$$
for all $x+iy\in \mathbf{B}_\eta(z_0)$. We recall that the  Bergman distance $d$ and the Euclidean distance $d_{Eucl}$ are equivalent on compact sets of $\C^n$ contained in $T_\Omega.$ So there exists a constant $R>1$ such that 
$$\frac{1}{R}d(X+iY,ie)<d_{Eucl}(X+iY,ie)<Rd(X+iY,ie)$$
for all $X+iY\in\overline{\mathbf{B}_1(ie)}$. The proof of the lemma follows from the following equivalence $$d_{Eucl}(X_1+iY_1,X_2+iY_2)\thickapprox\max(\|X_1-X_2\|,\|Y_1-Y_2\|)$$
and the equivalence given in Lemma 2.2 between $d_\Omega$ and the Euclidean distance $||\cdot||$ in $\mathbb R^n$ on compact sets of $\mathbb R^n$ contained in $\Omega.$
\end{proof}

The starting point of our analysis is the following  Whitney decomposition of the cone $\Omega$ which was obtained e.g. in \cite{BBGR, BBGNPR}.

\begin{lemma}\label{lemma7}
There is a positive integer $N$ such that given $\delta\in (0,1)$, one can find a sequence of points $\{y_j\}_{j=1,2,\cdots}$ in $\Omega$ with the following property:
\begin{enumerate}[\upshape (i)]
\item the balls $B_{\frac{\delta}{2}}(y_j)$ are pairwise disjoint;
\item the balls $B_\delta(y_j)$ cover $\Omega$;
\item each point of $\Omega$ belongs to at most $N$ of the balls $B_\delta(y_j)$
\end{enumerate}
\end{lemma}

\begin{defin}
The sequence $\{y_j\}$ is called a $\delta$-{\it lattice of} $\Omega.$
\end{defin}
%\begin{proof} See \cite{BBGR} for example.
%\end{proof}
Our goal is to obtain an atomic decomposition theorem for holomorphic functions in  $A_\s^{p,q}$ spaces. To this end, we need to derive a suitable version of the classical Whitney decomposition  of $\R^n$. Let $\{y_j\}$ be a $\delta$-{\it lattice of} $\Omega$ and let $g_j \in T$ be such that $g_j \cdot \textbf e = y_j.$ Let $R>1$ be a constant like in Lemma \ref{inclusion}.
%\section{Whitney decomposition with respect mixed norm}
%Here we give the whitney decompositions on the cone $\Omega$ and on $\R^n$.
%Referring to Coifman and Rochberg, the sequence $\{y_j\}$ is a $\delta$-lattice in $\Omega$.
We adopt the following notations: 
$$I_{l,j}=\{x\in\R^n:\,\,\,\|g_j^{-1}(x-x_{l,j})\|<\frac{\delta}{R}\}$$
$$I'_{l,j}=\{x\in\R^n:\,\,\,\|g_j^{-1}(x-x_{l,j})\|<\frac{\delta}{2R}\}$$
where  $\{x_{l,j}\}$ is a sequence in $\R^n$ to be determined.

From Lemma \ref{inclusion} we have immediately the following.
\begin{rem}\label{remark1}
For the  constant $R>1$ of Lemma \ref{inclusion}, the following inclusion holds
$$I_{l,j}+iB_{\frac \delta R} (y_j)\subset \textbf{B}_{\delta}(x_{lj}+iy_j).$$
\end{rem} 
\begin{lemma}\label{lemma8}
Let $\delta \in (0, 1).$ There exist a positive constant $R>1,$ a positive integer $N$ and a sequence  of points $\{x_{l,j}\}_{l\in\Z,\hskip 1truemm j\in\N}$ in $\R^n$ such that the following hold.
\begin{enumerate}[\upshape (i)]
\item $\{I_{l,j}\}_l$ form  a cover of $\R^n$;
\item $\{I'_{l,j}\}_l$ are pairwise disjoint;
\item  for each $j,$ every point of $\R^n$ belongs to at most $N$ balls $I_{l,j}$.
\end{enumerate}
\end{lemma}
\begin{proof}
Fix $j$ in $\N$ and define the collection $\mathcal A_j$ of sets in $\mathbb R^n$ by
$$\mathcal{A}_{j}=\left\{A\subset\mathbb{R}^{n}:\,\,\forall\, x,\,\,y\in A\,\,\textrm{distinct}\,\,\, \|g_{j}^{-1}(x-y)\|\geq \frac{\delta}{R}\right\}.$$
Clearly the collection $\mathcal{A}_{j}$ is non empty. Indeed the sets $\{\frac{\delta}{R} y_{j}, 0_{\R^n}\}$ are members of $\mathcal{A}_{j}$. Furthermore, the collection $\mathcal A_j$ is partially ordered with respect to inclusion.\\
Let $\mathcal C$ be a totally ordered subcollection of $\mathcal A_j.$ We set $F=\underset{A\in \mathcal C}{\cup} A.$ Given two distinct elements $x, y$ of $F,$ there are two members $A_1$ and $A_2$ of $\mathcal C$ such that $x\in A_1$ and $y\in A_2.$ But either $A_1 \subset A_2$ or $A_2 \subset A_1.$ So we have either $x, y \in A_1$ or $x, y \in A_2.$ Hence $||g_j^{-1} (x-y)|| \geq \frac \delta R.$ This shows that $F$ is a member of $\mathcal A_j.$ In other words, the collection $\mathcal A_j$ is inductive. An application of  Zorn's lemma then gives that the collection $\mathcal A_j$ has a maximal member $E_j.$ We write $E_j=\{x_{l,j}\}_{l\in L_{j}}.$ 

To prove assertion (ii), consider $l$ and $k$ such that $l\neq k$ and assume that $I'_{l,j}\cap I'_{k,j}$ contain at least an element $x$. Then
$$\|g_{j}^{-1}(x_{l,j}-x_{k,j})\|\leq \|g_{j}^{-1}(x_{l,j}-x)\| + \|g_{j}^{-1}(x-x_{k,j})\|<\frac{\delta}{2R}+ \frac{\delta}{2R}=\frac \delta R.$$ This would contradict the property that $\{x_{l,j},\,\,x_{k,j}\}$ is a subset of 
${E}_j,$ which is a member of $\mathcal A_j$. 
 
For assertion (i), let us suppose $\underset{l\in L_{j}}{\cup} I_{l,j}\neq \mathbb{R}^{n}$. Then there exists $\xi_{j}\in\mathbb{R}^{n}$ such that $\xi_{j}\notin\underset{l\in
L_{j}}{\cup}I_{l,j}$. Clearly the set $E_{j}\cup\{\xi_{j} \}$
is a member of $\mathcal{A}_{j}$. This would contradict the maximality of $E_j$ in $\mathcal A_j$. This completes the proof of assertion (i).

To prove assertion (iii), 
%let $z=x+iy\in T_\Omega$  such that $y\in B_\delta(y_j)$.
we fix $j.$ Given $x\in \mathbb R^n,$
it follows from assertion (i) that there exists a subset $L_j (x)$ of $L_j$ such that $$x\in\underset{l\in L_j(x)}{\cap}I_{l,j}.$$ 
We will show that there is a positive integer $N$ independent of $\delta$ such that $Card \hskip 2truemm L_j (x)\leq N$ for all $j\in \mathbb N$ and $x\in \mathbb R^n.$
It follows from Lemma 2.4 and Lemma 2.5 that for every $l\in L_{j},$
$${\mathbf B}_{\frac \delta {2R^2}} (x_{l, j} +iy_j) \subset I'_{l, j} \times B_{\frac \delta {2R}} (y_j).$$
So the balls ${\mathbf B}_{\frac \delta {2R^2}} (x_{l, j} +iy_j), \hskip 2truemm l\in L_{j}$ are pairwise disjoint since $R>1.$ Moreover, for every  $l\in L_j(x),$ we have
$$
\begin{array}{clcr}
{\mathbf B}_{\frac \delta {2R^2}} (x_{l, j} +iy_j) &\subset \{\xi +i\sigma \in T_\Omega: \|g_{j}^{-1}(\xi-x)\| <\frac{3\delta}{2R} \hskip 2truemm {\rm and} \hskip 2truemm \sigma \in B_{\frac {\delta} {2R}} (y_j)\}\\
& \subset {\mathbf B}_{\frac {3\delta}2} (x+iy_j).
\end{array}
$$ 
For the first inclusion, we applied the triangle inequality. We obtain
$$\underset{l\in L_{j}(x)}{\cup}\mathbf{B}_{\frac{\delta}{2R}}(x_{l,j}+iy_j)\subset\mathbf{B}_{ \frac{3\delta R}{2} }(x+iy_j).$$
We call $m$ the invariant measure on $T_\Omega$ given by
$$dm(\xi+i\sigma)={\Delta^{-\frac{2n}{r}}(\sigma)}d\xi d\sigma.$$
We conclude that
$$Card \hskip 1truemm L_j (x) \leq \frac {m\left (\mathbf B_{\frac {3\delta R}2} (i\mathbf e)\right)}{m\left (\mathbf B_{\frac {\delta}{2R}} (i\textbf e)\right )},$$
because
$$m\left (\underset{l\in L_{j}(x)}{\cup}\mathbf{B}_{\frac{\delta}{2R}}(x_{l,j}+iy_j)\right )=Card\hskip 1truemm L_j (x) \times m\left (\mathbf{B}_{\frac{\delta}{2R}}(i\mathbf e)\right ) \leq m\left ({\mathbf B}_{\frac {3\delta R}{2}} (i\mathbf e)\right ).$$
We finally prove that the collection $\{I_{l, j}\}_{l, j}$ is countable.  It suffices to show that for each $j,$ the collection $\{I_{l, j}\}_{l\in L_j}$ is countable. We fix $j.$ To every set $I'_{l, j},$ we assign a point of $\mathbb Q^n$ belonging to $I'_{l, j}.$ Since $\bigcap \limits_l I'_{l, j} = \emptyset,$ this defines a one-to-one correspondence from the collection $\{I'_{l, j}\}_{l\in L_j}$ to a subset of $\mathbb Q^n.$ This shows that the collection $\{I'_{l, j}\}_{l\in L_j}$ is at most countable. Moreover the collection $\{I_{l, j}\}_{l\in L_j}$ which has the same cardinal as the collection $\{I'_{l, j}\}_{l\in L_j}$ is infinite: the proof is elementary since $\bigcup \limits_{l\in L_j} I_{l, j}=\mathbb R^n$ is unbounded. The proof of the lemma is complete.

%It suffices to show that for each $j,$ the collection $\{I_{l, j}\}_{l\in L_j}$ is countable. Since by assertion (iii), each point $\xi$ of $\mathbb Q^n$ belongs to at most $N$ sets $I_{l, j},$ the collection
%$\mathcal I_j$ of all sets $I_{l, j}$
%is countable and is an open covering of $\mathbb Q^n.$ The density of $\mathbb Q^n$ in $\mathbb R^n$ then implies that $\{I_{l, j}\}_{l, j}=\mathcal I_j.$

\end{proof}

\begin{rem}
We just proved in Lemma \ref{lemma8} that for each $j=1, 2,\cdots,$ the index set $L_j$ is countable. In analogy with the one-dimensional case \cite{RT}, we took $L_j = \mathbb Z$ in the statement of Lemma \ref{lemma8} and in the statement of Theorem A.
\end{rem}

\subsection{A $\delta$-lattice in $T_\Omega$}
\begin{defin}
The sequence $\{z_{lj}=x_{lj} +iy_j\}_{l\in \mathbb Z, j\in \mathbb N}$ defined in Lemma \ref{lemma8} will be called a $\delta$-{\it lattice \hskip 1truemm in} $T_\Omega$.
\end{defin}

We have the following lemma.

\begin{lemma}
Let $\{z_{l,j}=x_{l,j}+iy_j\}_{l\in \mathbb Z, j\in \mathbb N}$ be a $\delta$-lattice in $T_\Omega$. There exists a positive constant $C=C(\delta, R)$ such that for all $l\in \mathbb Z$, $j\in \mathbb N,$ the following hold.
\begin{enumerate}[\upshape (a)]
\item $$\int_{I_{l,j}}dx\leq C\Delta^\frac{n}{r}(y_j).$$
\item $$\int_{\R^n}\sum_{l\in L_j}\chi_{\{x\in I_{l,j}:\,\,\textbf{d}(x+iy,w)<1\}}(x)dx\leq C\Delta^\frac{n}{r}(y_j),\,\,\forall y\in B_\delta(y_j), \forall w\in T_\Omega .$$
\end{enumerate}
\end{lemma}
\begin{proof}
We denote $Det$ the usual determinant of an endomorphism of $\mathbb R^n.$
\begin{enumerate}
\item[(a)] 
We set $u=g_j^{-1}(x-x_{lj}).$ Then 
\begin{align*}
\int_{I_{l,j}}dx&=\int_{\|u\|< \frac{\delta}{R}} Det\hskip 1truemm  (g_j)du\\
&=\Delta^\frac{n}{r}(y_j)\int_{\|u\|<\frac{\delta}{R}}du=C\Delta^\frac{n}{r}(y_j).
\end{align*} 
For the second equality, we applied the formula (\ref{Delta}). This proves assertion (a). 
\item[(b)]
By assertion (iii) of Lemma \ref{lemma8}, we have
$$\sum_{l\in L_j}\chi_{I_{l,j}} (x) \leq N$$
for every $j.$
Then 
\begin{align*}
\int_{\R^n}\sum_{l\in L_j}\chi_{\{x\in I_{l,j}:\,\,\textbf{d}(x+iy,w)<1\}}(x)dx&\leq N\int_{\{x\in \mathbb R^n: \hskip 1truemm x+iy \in \textbf B_1 (w)\}} dx
%&\leq NC\delta^\frac{n}{r}(y_j).
\end{align*}
We set $w=u+iv.$ By Lemma 2.4, we have the implication
$$x+iy\in {\textbf B}_1 (w) \Rightarrow ||g^{-1} (x-u)||<R \hskip 2truemm {\rm and} \hskip 2truemm y\in B_R (v)$$
with $\hskip 2truemm g\cdot \textbf e = v.$
So
$$\int_{\{x\in \mathbb R^n: \hskip 2truemm x+iy \in \textbf B_1 (w)\}} dx \leq \int_{\{x\in \mathbb R^n: \hskip 2truemm ||g^{-1} (x-u)||<R\}} dx = CDet\hskip 1truemm  (g) = C\Delta^{\frac nr} (v).$$
But $d_\Omega (y, y_j) < \delta$ and $d_\Omega (y, v) < R.$ This implies that $d(v, y_j) < \delta + R.$ Henceforth $\Delta^{\frac nr} (v) \leq C \Delta^{\frac nr} (y_j)$ by Lemma 2.1. This gives assertion (b).
\end{enumerate}
\end{proof} 

\section{Atomic decomposition}
\subsection{The sampling theorem}
We first record the following lemma (See e.g. \cite{BBGNPR}). 
\begin{lemma}\label{lemma10}
Let $1\leq p < \infty.$ Given $\delta\in (0,1)$, there exists a positive constant $C$ such that, for each holomorphic function $F$ in $T_\Omega$ we have 
\begin{enumerate}[\upshape (i)]
\item $|F(z)|^p\leq C\delta^{-2n}\int_{\textbf {B}_\delta(z)}|F(u+iv)|^p\frac{du\,dv}{\Delta^\frac{2n}{r}(v)}$;
\item if  $d (z,\zeta)<\delta$ then
$$|F(z)-F(\zeta)|^p\leq C\delta^{p}\int_{\textbf B_1(z)}|F(u+iv)|^p\frac{du\,dv}{\Delta^\frac{2n}{r}(v)}$$
\end{enumerate}
\end{lemma}
%\begin{proof} 
%See \cite{BBGNPR} for instance.
%\end{proof}
  
%Given a function $F$ on $T_\Omega$ and $y\in \Omega$, we introduce the partial function 
%$$\left 
%\begin{array}{clcr}
%F_y :&\mathbb R^n \rightarrow \mathbb C\\
%&x\mapsto F_y (x) = F(x+iy) 
%\end{array}
%\right
%.
%$$

For the second lemma, the reader should refer to \cite{BBGR}, Lemma 4.5.

\begin{lemma}\label{lemma11}
Suppose $\delta\in(0,1)$ and $1\leq p,q<\infty$. There exists a positive constant $C$ such that
\begin{equation}
\|F(\cdot+iy)\|^q_p\leq C\int_{B_\delta(y)}\|F(\cdot+iv)\|^q_p\frac{dv}{\Delta^\frac{n}{r}(v)}
\end{equation}  
for every holomorphic function $F$ on $T_\Omega$ and every $y\in\Omega$.  
\end{lemma}

The following is our sampling theorem.
\begin{thm}\label{sampling}
Let $\delta \in (0, 1)$ satisfy the assumption of Corollary 2.3 and let $\{z_{l,j}=x_{l,j}+iy_j\}_{l\in\Z, \hskip 1truemm j\in\N}$ be a $\delta$-lattice in $T_\Omega$. Let $1\leq p,q<\infty$ and let  $\textbf s \in \mathbb R^r$ be such that $s_k > \frac {n_k}2, \hskip 2truemm k=1,\cdots,r.$  There exists a positive constant $C_\delta = C_\delta (\s, p, q)$ such that for every $F\in A_\s^{p,q}$, we have
\begin{equation}
\sum_j\left(\sum_l|F(z_{l,j})|^p\right)^\frac{q}{p}\Delta_{\s+\frac{nq}{rp}}(y_j)\leq C_\delta\|F\|_{A_\s^{p,q}}^q
\end{equation} 
Moreover, if $\delta$ is small enough, the converse inequality
\begin{equation}\label{eqq}
\|F\|_{A_\s^{p,q}}^q\leq C_\delta \sum_j\left(\sum_l|F(z_{l,j})|^p\right)^\frac{q}{p}\Delta_{\s+\frac{nq}{rp}}(y_j)
\end{equation}
is also valid.
\end{thm}
\begin{proof}
From Lemma \ref{lemma10} we have 
\begin{equation}\label{equation5}
|F(z_{l,j})|^{p}\leq C\delta
^{-2n}\int_{{\textbf B}_ \frac{\delta}{2R^2}(z_{l,j})}|F(u+iv)|^p\,\frac{du\,dv}{\Delta^{\frac{2n}{r}}(v)}.
\end{equation}
It follows from the inclusion $\textbf{B}_ \frac{\delta}{2R^2}(z_{l,j})\subset\left\{u+iv:\,\,u\in I'_{l,j},\,\,v\in B_\frac{\delta}{2R}(y_j)\right\}$  that
\begin{equation}\label{equation6}
|F(z_{l,j})|^{p}\leq C\delta ^{-2n}\int _{I'_{l,j} }\,du\int_{ B_\frac{\delta}{2R}(y_j)
}|F(u+iv)|^p\,\frac{dv}{\Delta^\frac{2n}{r}(v)}.
\end{equation}
From the equivalence of $\Delta(v)$ and $\Delta(y_j)$ whenever $v\in B_\frac{\delta}{2R}(y_j),$ we obtain that 
\begin{equation}\label{equation6'}
|F(z_{l,j})|^{p}\leq \frac{C\delta ^{-2n}}{\Delta^\frac{2n}{r}(y_j)}\int _{I'_{l,j} }\,du\int_{ B_\frac{\delta}{2R}(y_j)
}|F(u+iv)|^p\,dv.
\end{equation}
Next, a successive application of Lemma \ref{lemma8}, Corollary 2.3 and the non-increasing property of the function $\Omega\ni v\mapsto\|F(\cdot+iv)\|_p^p$  gives the existence of a positive constant $\gamma$ such that
\begin{align*}
\sum \limits_{l\in L_j}|F(z_{l,j})|^{p}&\leq \frac{C\delta ^{-2n}}{\Delta^\frac{2n}{r}(y_j)}\int _{\R^n }\,du\int_{ B_\frac{\delta}{2R}(y_j)
}|F(u+iv)|^p\,dv\\
&=\frac{C\delta ^{-2n}}{\Delta^\frac{2n}{r}(y_j)}\int_{ B_\frac{\delta}{2R}(y_j)
}\|F(\cdot+iv)\|_p^p\,dv\\
&\leq \frac{C\delta ^{-2n}}{\Delta^\frac{2n}{r}(y_j)}\int_{ B_\frac{\delta}{2R}(y_j)
}\|F(\cdot+i\gamma y_j)\|_p^p\,dv\\
&\leq \frac{C\delta ^{-2n}}{\Delta^\frac{n}{r}(y_j)}\|F(\cdot+i\gamma y_j)\|_p^p.
\end{align*}
Finally,  we obtain
\begin{equation}\label{equation7}
\sum_j\left(\sum_l|F(z_{l,j})|^p\right)^\frac{q}{p}\Delta_{\s+\frac{nq}{rp}}(y_j)\leq C_\delta^{\frac qp} \sum_j\|F(\cdot+i\gamma y_j)\|_p^q\Delta_\s (y_j).
\end{equation} 
We define the holomorphic function $F_\gamma$ by
$$F_\gamma(x+iy)=F(\gamma(x+iy)).$$
By Lemma \ref{lemma11}, we get 
\begin{equation}\label{equation8}
\|F(\cdot+i\gamma y_j)\|_p^q =\gamma^{\frac {nq}p}\|F_{\gamma}(\cdot +iy_{j})\|_{p}^{q}\leq C\gamma^{\frac {nq}p}\int
  _{ B_\frac{\delta}{2R^2}(y_j) }\|F_{\gamma}(\cdot
  +iy)\|_{p}^{q}\frac{dy}{\Delta^{\frac{n}{r}}(y)}.
\end{equation}
It follows from (\ref{equation8}), Lemma \ref{lemma7} and the equivalence of $\Delta(y)$ and $\Delta(y_j)$ whenever $y\in B_\frac{\delta}{2R^2}(y_j)$  that
\begin{eqnarray*}
\underset{j}{\sum}\|F (\cdot
  +i\gamma y_{j})\|_{p}^{q}\Delta_{\s}(y_{j})&\leq C\gamma^{\frac {nq}p}\int _{\Omega}\|F_\gamma (\cdot +iy)
  \|_{p}^{q}\Delta_{\s -\frac{n}{r}}(y)\,dy\\
&=C\int_\Omega ||F(\cdot +i\gamma y)||_p^q \Delta_{\s -\frac nr} (y)dy.
\end{eqnarray*}
Moreover, taking $v=\gamma y$ we obtain
\begin{equation}\label{equation9}
\underset{j}{\sum}\|F(\cdot
  +i\gamma y_{j})\|_{p}^{q} \Delta_{\s}(y_{j}) \leq C(\gamma, \s, p, q)\int _{\Omega}\|F(\cdot
  +iv)\|_{p}^{q}\Delta_{\s -\frac{n}{r}}(v)\,dv.
\end{equation}
So the estimate (3.2) is a direct consequence of {(\ref{equation7})} and {(\ref{equation9})}.

Conversely, a successive application of Lemma \ref{lemma8}, the triangle inequality and assertion a) of Lemma 2.11 gives
\begin{eqnarray*}
\|F(\cdot+iy)\|_p^p &\leq C_p\left \{\underset{l\in L_j}{\sum}\int
    _{I_{l,j}}|F(x+iy)-F(z_{l,j})|^{p}\,dx+
    \underset{l\in L_j}{\sum}|F(z_{l,j})|^p\int_{I_{l,j}}\,dx\right \}\\
&\leq C_p\left \{\underset{l\in L_j}{\sum}\int
    _{I_{l,j}}|F(x+iy)-F(z_{l,j})|^{p}\,dx+
    \underset{l\in L_j}{\sum}|F(z_{l,j})|^p\Delta^{\frac nr}(y_j)\right \}.
\end{eqnarray*}
for all $y\in \Omega$.
In the sequel, for fixed $y \in \Omega,$ we set
$$K_j(w)=\int_{\R^n}\sum_{l\in L_j}\chi_{\{x\in I_{l,j}:\,\,d(x+iy,w)<1\}}(x)dx$$
and we write
$$
N_{p,q}(F) = \int_{y\in\Omega}\underset{j\in\mathbb{N}}{\sum}\chi_{ B_\delta(y_j)}(y)\times$$
$$\left(\int_{v\in\Omega}\int_{\mathbb{R}^{n}}
K_j(u+iv) |F(u+iv)|^{p}\chi_{d_\Omega (y, v) < R}\frac{du\,dv}{\Delta (v)^{\frac{2n}{r}}}
\right)^{\frac{q}{p}}\Delta_{\s-\frac{n}{r}}(y)dy.
$$
Using assertion (ii) of Lemma 3.1, we obtain easily that
\begin{align*}
\|F\|_{A_{\s}^{p,q}}^{q}&\leq \int_{\underset{j}{\cup}B_\delta(y_j)}\|F(\cdot
+iy)\|_{p}^{q}\Delta _{\s -\frac{n}{r}}(y)dy\\
&\leq C_{p,q}\delta
^{q} N_{p,q}(F)+C_{p,q}\underset{j} {\sum} \left(
\underset{l}{\sum}|F(z_{l,j})|^{p}\right)^{\frac{q}{p}}\Delta _{
{\s}+\frac{nq}{rp}}(y_{j}).
\end{align*}
To prove (\ref{eqq}) it suffices to establish the following inequality:
$$N_{p,q}(F)\leq C \|F\|_{A_{\s}^{p,q}}^{q}.$$
To this end, first observe that by assertion (b) of Lemma 2.11, we have
$$K_j(w)  \leq C\Delta^{\frac nr} (y_j), \quad \forall y\in B_\delta (y_j), \hskip 2truemm \forall w\in \Omega.$$
Now by Lemma 2.1, we have the equivalence $\Delta(v)\sim\Delta(y_j)\sim \Delta(y)$ whenever $v\in B_R (y)$ and $y\in B_\delta (y_j)$ with equivalence constants independent of $\delta.$ This combined with an  application of
assertion (iii) of Lemma \ref{lemma7}  gives that 
$$
N_{p,q}(F)\leq
CN\int_\Omega
\left(\int_{d(v,y)<R}\|F(\cdot+iv)\|_p^p\frac{dv}{\Delta^\frac{n}{r}(v)}\right)^\frac{q}{p}\Delta_{\s-\frac{n}{r}}(y)dy.
$$ 
 Next, from the non-increasing property of the mapping $v\in\Omega\mapsto\|F(\cdot+iv)\|_p,$ Corollary \ref{lemma5} and the $G$-invariance of the measure $\frac {dv}{\Delta^{\frac nr} (v)}$ on $\Omega,$ there exists a positive constant $\gamma$ independent of $\delta$ such that
\begin{equation*}
N_{p,q}(F)\leq CN\int_\Omega\|F(\cdot+i\gamma y)\|_p^q  \Delta_{\s-\frac{n}{r}}(y)dy.
\end{equation*}
Finally, taking $t=\gamma y$ on the right hand side of the previous inequality, we obtain that
$$N_{p,q}(F)\leq C (\gamma)\|F\|_{A_{\s}^{p,q}}^{q}.$$
\end{proof}

\subsection{Proof of Theorem A}
We can now prove the atomic decomposition theorem (Theorem A). Here is its more precise statement.

\begin{thm}\label{thA'}
Let $\delta \in (0, 1)$ and let $\{z_{l,j}=x_{l,j}+iy_j\}_{l\in\Z, \hskip 1truemmj\in\N}$ be a $\delta$-lattice in $T_\Omega.$ Let  $\textbf s$ be a vector of $\R^r$ such that $s_k > \frac {n_k}2, \hskip 1truemm k=1,\cdots,r.$ Assume that $P_\s$ extends to a bounded operator on $L_\textbf s^{p,q}$. Then there exists a positive constant $C$ such that the following two assertions hold.
\begin{enumerate}[\upshape (i)]
\item For every sequence $\{\lambda_{l,j}\}_{l\in\Z,  \hskip 1truemm j\in\N}$ such that $$\sum_j\left(\sum_l|\lambda_{l,j}|^p\right)^\frac{q}{p}\Delta_{\textbf s+\frac{nq}{rp}}(y_j)<\infty,$$ the series $$\sum_{l,j}\lambda_{l,j}\Delta_{\textbf s+\frac{nq}{rp}}(y_j)B_\s(z,z_{l,j})$$ is convergent in $A_\textbf s^{p,q}$. Moreover, its sum $F$ satisfies the inequality
$$\|F\|_{A_\textbf s^{p,q}}^q\leq C_\delta \sum_j\left(\sum_l|\lambda_{l,j}|^p\right)^\frac{q}{p}\Delta_{\textbf s+\frac{nq}{rp}}(y_j)$$
\item For $\delta$ small enough, every function $F\in A_\textbf s^{p,q}$ may be written as 
$$F(z)=\sum_{l,j}\lambda_{l,j}\Delta_{\textbf s+\frac{nq}{rp}}(y_j)B_\textbf s (z,z_{l,j}),$$ 
with $$\sum_j\left(\sum_l|\lambda_{l,j}|^p\right)^\frac{q}{p}\Delta_{\textbf s+\frac{nq}{rp}}(y_j)\leq C_\delta \|F\|_{A_\textbf s^{p,q}}^q$$
\end{enumerate} 
\end{thm}
 
\begin{proof}[Proof of Theorem \ref{thA'}]
 Let $p\in [1, \infty], \hskip 2truemm q\in(1,\infty), $ and call $p'$ and $q'$ their conjugate exponents, i.e  $\frac{1}{p}+\frac{1}{p'}=1$ and $\frac{1}{q}+\frac{1}{q'}=1$. Let $\s \in \mathbb R^r$ such that $s_k > \frac {n_k}2, \hskip 2truemm k=1,\cdots, r.$ Recall that (cf. \cite{DD}) if $P_\s:L^{p',q'}_\s\to A^{p',q'}_\s$ is bounded, then the dual space of $A^{p',q'}_\s$ identifies with $A^{p,q}_\s$ with respect to the pairing
$$<F, G>_{\textbf s} = \int_{T_\Omega} F(x+iy)\overline {G(x+iy)}\Delta_{\textbf s - \frac nr} (y)dxdy.$$
Denote by $l^{p,q}_\s$ the space of complex sequences $\{\lambda_{l,j}\}_{l\in\Z, \hskip 1truemmj\in\N}$ such that
$$||\{\lambda_{l,j}\}||_{l^{p,q}_\s} = \left (\sum_j \left (\sum_l|\lambda_{l,j}|^p\right )^\frac{q}{p}\Delta_{\s+\frac{nq}{rp}}(y_j)\right )^{\frac 1q}<\infty.$$
We have the duality $l^{p,q}_\s =(l_{\s}^{p' ,q' })'$ with respect to the pairing
$$<\lambda,\,\,\mu>_{l_{\s}^{p',q'},\,\,l_{\s}^{p,q}}
=\underset{l,j}{\sum}\lambda_{l,j}
\overline{\mu}_{l,j}\Delta_{\s+\frac{n}{r}}(y_j).$$
Then from the first part of the sampling theorem,  the operator 
\begin{eqnarray*}
R:&A_{\s}^{p',q' }&\to l_{\s}^{p' ,q' }\\
&F&\mapsto  RF=\{F(z_{l,j})\}_{l\in\Z,\,j\in\mathbb{N}}
\end{eqnarray*}
is bounded. So the adjoint operator $R^*$ of $R$ is also a bounded operator  from $l^{p,q}_\s$ to $A^{p,q}_\s$. Its explicit formula is 
 $$R^*(\{\lambda _{l,j}\})(z)=\sum _{l,j}\lambda
_{l,j}\Delta_{\s +\frac{n}{r}} (y_{j})B_{\s}(z, z_{l,j}).$$
This completes the proof of assertion (i).

From the second part of the sampling theorem, if $\delta$ is small enough, the adjoint operator $R^*:l^{p,q}_\nu\to A^{p,q}_\nu$ of $R$ is onto. Moreover,  we call $\mathcal{N}$ the subspace of $l^{p,q}_\nu$ consisting of all sequences $\{\lambda_{l,j}\}_{l\in\Z,\,j\in\mathbb{N}}$ such that the mapping $$z\mapsto\sum_{l,j}\lambda_{l,j} \Delta_{\s+\frac{n}{r}}(y_j)B_\s (z,z_{l,j})$$
vanishes identically. Then the linear operator
\begin{eqnarray*}
\varphi:l_{\textbf s}^{p,q}/\mathcal{N}&\to& A_{\textbf s}^{p,q}\\
\{\lambda _{l,j}\}&\mapsto &\sum _{l,j}\lambda
_{l,j}B_{\nu}(z,z_{l,j})\Delta_{\s +\frac{n}{r}} (y_{j})
\end{eqnarray*}   
is a bounded isomorphism from the Banach quotient space $l_{\nu}^{p,q}/\mathcal{N}$ to $A_{\textbf s}^{p,q}$.  The inverse operator $\varphi^{-1}$ of $\varphi$ is continuous. This gives assertion (ii).
\end{proof}
%%%%%%%%%%%%%%%%%%%%%%%%%%%%%%%%%%%%%%%%%%%%%%%%%%%%%%%
%%%%%%%%%%%%%%%%%%%%%%%%%%%%%%%%%%%%%%%%%%%%%%%%%%%%%%%
\section{Interpolation}
In this section we determine the  interpolation space via the complex method between two mixed norm  weighted Bergman spaces. 
\subsection{Interpolation via the complex method between Banach spaces}
Throughout this section we denote by $S$ the open strip in the complex plane defined by
$$S=\{z=x+iy\in\C:\,\,\,0<x<1\}.$$
Its closure $\overline S$ is
$$\overline S=\{z=x+iy\in\C:\,\,\,0\leq x\leq 1\}.$$ 
%We have $\partial S = L\cup R$ with
%$$L=\{z=x+iy\in\C:\,\,\,x=0\},\,\,\,R=\{z=x+iy\in\C:\,\,\,x=1\}$$

Let $X_0$ and $X_1$ be two compatible Banach spaces, i.e. they are continuously embedded in a Hausdorff topological space. Then $X_0+X_1$ becomes a Banach space with the norm
$$\|f\|_{X_0+X_1}=\inf \hskip 1truemm \left (\|f_0\|_{X_0}+\|f_1\|_{X_1}\right ),\,\,\,\,f=f_0 +f_1, \hskip 1truemm f_0\in X_0, \hskip 1truemm f\in X_1\}.$$
%where the infimum is taken over all decompositions $f=f_0+f_1$ with $f_0\in X_0$ and $f_1\in X_1$. 
We will denote by $\mathcal{F}(X_0,X_1)$ the space of analytic mappings 
\begin{eqnarray*}
f:&\overline {S }&\rightarrow X_0+X_1\\
&\zeta &\mapsto f_\zeta
\end{eqnarray*}
 with the following properties:  
\begin{enumerate} [\upshape (1)]    
\item  $f$ is bounded and continuous on  $\overline S;$  
%that is analytic in $S$, continuous on $\bar S$, and bounded on $\bar S$. 
%\item $f_\theta=f$.
\item $f$ is analytic in $S;$
\item For $k=0, 1$ the function $y\mapsto f_{k+iy}$ is bounded  and continuous from the real line into $X_k$.
\end{enumerate}
The space $\mathcal{F}(X_0,X_1)$ is a Banach space with the following norm:
$$\|f\|_{\mathcal{F}}=\max \hskip 2truemm \left (\sup_{\Re e\hskip 1truemm \zeta =0}\|f_\zeta\|_{X_0},\sup_{\Re e\hskip 1truemm \zeta =1}\|f_\zeta\|_{X_1}\right ).$$
If $\theta\in (0,1)$, the complex interpolation space $[X_0,X_1]_\theta$ is the subspace of $\mathcal{F}(X_0,X_1)$ consisting of holomorphic functions $g$ on $T_\Omega$ such that $f_\theta=g$ for some $f\in \mathcal{F}(X_0,X_1)$. The space $[X_0,X_1]_\theta$ is a Banach space with the following norm:
$$\|g\|_\theta=\inf \{||f||_{\mathcal{F}(X_0,X_1)}: g=f_\theta\}.$$

Referring to \cite{BL} and \cite{SW} (cf. also \cite{Z}), the complex method of interpolation spaces is functorial in the following sense:  if $Y_0$ and $Y_1$ denote two other compatible Banach spaces of measurable functions on $T_\Omega,$ then if
$$T:\,\,X_0+X_1\to Y_0+Y_1$$
is a linear operator with the property that $T$ maps $X_0$ boundedly into $Y_0$ and $T$ maps $X_1$ boundedly into $Y_1$, then $T$ maps $[X_0,X_1]_\theta$ boundedly into $[Y_0,Y_1]_\theta$, for each $\theta\in (0,1)$. See \cite{BL} for more information about complex interpolation.\\
A classical example of  interpolation via the complex method concerns $L^{p,q}$ spaces with a change of measures. We state it in our setting of a tube domain $T_\Omega$ over a symmetric cone $\Omega.$

\begin{thm}\label{mixed} \cite{C, SW1}
Let $1\leq p_0, p_1, q_0, q_1 \leq \infty.$ Given two positive measurable functions (weights) $\omega_0, \hskip 2truemm \omega_1$ on $\Omega,$ then for every $\theta \in (0, 1),$ we have
$$[L^{q_0}\left ((\Omega, \omega_0 (y)dy); L^{p_0} (\mathbb R^n, dx)), L^{q_1}((\Omega, \omega_1 (y)dy); L^{p_1} (\mathbb R^n, dx)\right )]_\theta$$
$$=L^{q}((\Omega, \omega (y)dy); L^p (\mathbb R^n, dx))$$
with equal norms, provided that
$$\frac{1}{p}=\frac{1-\theta}{p_0}+\frac{\theta}{p_1}$$
$$\frac{1}{q}=\frac{1-\theta}{q_0}+\frac{\theta}{q_1}$$
$$\omega^\frac{1}{q}=\omega_0^\frac{1-\theta}{q_0}\omega_1^\frac{\theta}{q_1}.$$
\end{thm}

We finally record the Wolff reiteration theorem \cite{W, JNP} .

\begin{thm}\label{W}
Let $A_1, A_2, A_3, A_4$ be compatible Banach spaces. Suppose $[A_1, A_3]_\theta = A_2$ and $[A_2, A_4]_\varphi = A_3.$ Then 
$$[A_1, A_4]_\xi = A_2, \hskip 2truemm [A_1, A_4]_\psi = A_3$$
with $\xi = \frac {\theta \varphi}{1-\theta +\theta \varphi}, \psi = \frac {\varphi}{1-\theta +\theta \varphi}.$
\end{thm}

%To prove Theorem B by using atomic decomposition (see Theorem A) it suffices to refer to %\cite{BGN}. So, in this note we only prove Theorem B by using Wolff's abstract reiteration theorem. %To this end the starting point of our proof is some estimates of Bergman operators.
 
\subsection{Preliminary results on tube domains over symmetric cones}
We recall the following notations given in the introduction:\\
$$n_k = \frac {2(\frac nr -1)(k-1)}{r-1}$$
 and 
$$m_k = \frac {2(\frac nr -1)(r-k)}{r-1}$$
for every $k=1,\cdots,r.$ We recall the following two results (\cite{NT, BGN}). 

\begin{lemma}
Let $\s, \t\in \mathbb R^n$ be such that $s_k, t_k > \frac {n_k}2, \hskip 2truemm k=1, \cdots, r.$ Then the subspace $A^{2, 2}_{\t}\cap A^{p, q}_{\s}$ is dense in the weighted Bergman space $A^{p, q}_{\s}$ for all $1\leq p \leq \infty$ and $1\leq q <\infty$
\end{lemma}

\begin{cor}\label{cor}
Let $\s \in \mathbb R^n$ be such that $s_k > \frac {n_k}2, \hskip 2truemm k=1, \cdots, r.$ Assume that $\t\in \mathbb R^n$ and $1\leq p, q <\infty$ are such that $P_{\t}$ extends to a bounded operator on $L_{\s}^{p,q}.$ Then $P_{\t}$ is the identity on $A_{\s}^{p,q};$ in particular $P_{\t} (L_{\s}^{p,q})=A_{\s}^{p,q}.$
\end{cor}

The following theorem was  proved in \cite{BGN}.

\begin{thm}\label{main3}
Let $\s, \t \in \mathbb R^r$ and  $1\leq p, q\leq \infty$.  %Suppose that $\s, \Re e\hskip 1truemm \t>\textbf{d}$. 
Then the positive Bergman operator $P^+_\t$ defined by
$$P^+_\t f(\xi+i\tau)=d_\t\int_\Omega\left(\left|\Delta_{-\t-\frac{n}{r}}(\frac{\cdot+i(\tau+v)}{i})\right|*f(\cdot+iv)\right)(\xi)\Delta_{\t-\frac{n}{r}}(v)dv$$
is bounded on $L_{\s}^{p,q}$ when $t_k >\frac {n}r -1, \hskip 2truemm k=1,\cdots,r$ and
$$\max \limits_{1\leq k \leq r} \left (1, \frac {s_k - \frac {n_k}2 +\frac {m_k}2}{t_k - \frac {n_k}2}\right ) < q< 1+ \min \limits_{1\leq k \leq r}\left (1, \frac {s_k - \frac {n_k}2}{\frac {m_k}2}\right ) \quad {\rm if} \hskip 2truemm q>1$$
$$\left [{\rm resp.} \hskip 2truemm s_k > \frac {n_k}2 \quad {\rm and} \quad t_k - s_k > \frac {m_k}2 \quad {\rm if} \quad q=1 \right].$$
In this case, $P_\t$ extends to a bounded operator from $L_{\s}^{p,q}$ onto $A_{\s}^{p,q}.$
\end{thm}

\begin{proof}
For $q>1,$ the first part of this theorem is just the case $\alpha = 0$ in Theorem 3.8 of \cite{BGN} for symmetric cones. The case $q=1$ is an easy exercise (cf. e.g. Theorem II.7 of \cite{BT}). The proof of the surjectivity of $P_\t$ uses the previous corollary. 
\end{proof}

\subsection{Proof of Theorem B}
%Now we are ready to give the proof of the interpolation theorem.

%\begin{proof}[Proof of Theorem B] 
(1)
We adopt the following notations:
$$||g||_\theta = ||g||_{\left [L^{p_0, q_0}_{\s_0}, \hskip 1truemm L^{p_1, q_1}_{\s_1}\right ]_\theta}$$
and
$$||g||^{anal}_\theta = ||g||_{\left [A^{p_0, q_0}_{\s_0}, \hskip 1truemm A^{p_1, q_1}_{\s_1}\right ]_\theta}.$$
It suffices to show the existence of a positive constant $C$ such that the following two estimates are valid.
\begin{equation}\label{embed1}
 ||g||^{anal}_\theta \leq C||g||_{A^{p, q}_{\s}} \quad \quad \forall g\in A^{p, q}_{\s}.
\end{equation}
\begin{equation}\label{embed2}
||g||_{A^{p, q}_{\s}}\leq ||g||^{anal}_\theta \quad \quad \forall g\in \left [A^{p_0, q_0}_{\s_0}, A^{p_1, q_1}_{\s_1}\right ]_\theta;
\end{equation} 

We first the estimate (\ref{embed1}). By Theorem 4.1, we have
$$[L^{p_0, q_0}_{\s_0}, L^{p_1, q_1}_{\s_1}]_\theta=L^{p, q}_{\s}$$
with equivalent norms, provided that
$$\frac{1}{p}=\frac{1-\theta}{p_0}+\frac{\theta}{p_1}$$
$$\frac{1}{q}=\frac{1-\theta}{q_0}+\frac{\theta}{q_1}$$
$$\frac{\s}{q}=\frac{(1-\theta)\s_0}{q_0}+\frac{\theta \s_1}{q_1}.$$
In particular,  for every $g\in L^{p, q}_{\s},$ we have
\begin{equation}\label{lebesgue}
||g||_{L^{p, q}_{\s}} \simeq ||g||_\theta = \inf \left \{||f||_{\mathcal F\left (L^{p_0, q_0}_{\s_0}, \hskip 1truemm L^{p_1, q_1}_{\s_1}\right )}: g=f_\theta\right \}.
\end{equation}
%e check easily that $1\leq q < q_{\s}.$ 
By Theorem 4.5, for $\t$ large (i.e. each $t_k, \hskip 2truemm k=1,\cdots,r$ is large), the weighted Bergman projector $P_{\t}$ extends to a bounded operator from $L^{p_i, q_i}_{\s_i}$ onto $A^{p_i, q_i}_{\s_i}, \hskip 2truemm i=0, 1$ and hence from $L^{p, q}_{\s}$ onto $A^{p, q}_{\s}.$ Then by Corollary \ref{cor}, for every $g\in A^{p_i, q_i}_{\s_i}, \hskip 2truemm i=0, 1$ and for every $g\in A^{p, q}_{\s},$ we have $P_{\t} g = g.$\\
Now let $g\in A^{p, q}_{\s}.$ For $f\in \mathcal F(L^{p_0, q_0}_{\s_0}, L^{p_1, q_1}_{\s_1}),$  we define the mapping
$$
P_{\t}\circ f: \hskip 2truemm \overline {S} \rightarrow A^{p_0, q_0}_{\s_0}+ A^{p_1, q_1}_{\s_1}
$$
by $(P_{\t}\circ f)_\zeta = P_{\t}\circ f_\zeta.$ Then
 $P_{\t} \circ f \in \mathcal F(A^{p_0, q_0}_{\s_0}, A^{p_1, q_1}_{\s_1})$ and if  $f_\theta = g,$ we have $(P_{\t} \circ f )_\theta = P_{\t} \circ f_\theta = P_{\t} g = g.$ So
$$
\begin{array}{clcr}
||g||^{anal}_\theta &:= \inf \hskip 1truemm \{||\varphi||_{\mathcal F\left (A^{p_0, q_0}_{\s_0}, \hskip 1truemm A^{p_1, q_1}_{\s_1}\right )}: g=\varphi_\theta\}\\
&\leq \left \Vert P_\t \circ f \right \Vert_{\mathcal F\left (A^{p_0, q_0}_{\s_0}, \hskip 1truemm A^{p_1, q_1}_{\s_1}\right )}\\
&:= \max \hskip 1truemm \left \{\sup \limits_{\Re e \hskip 1truemm \zeta =0} ||(P_\t \circ f)_\zeta||_{A^{p_0, q_0}_{\s_0}}, \sup \limits_{\Re e \hskip 1truemm\zeta =1} ||(P_\t \circ f)_\zeta||_{A^{p_1, q_1}_{\s_1}}\right \}
\end{array}
$$
for every $f\in \mathcal F(L^{p_0, q_0}_{\s_0}, L^{p_1, q_1}_{\s_1})$ such that $f_\theta = g.$ By Theorem \ref{main3}, we get
$$||g||^{anal}_\theta \leq C_\t \inf \hskip 1truemm \{||f||_{\mathcal F(L^{p_0, q_0}_{\s_0}, \hskip 1truemm L^{p_1, q_1}_{\s_1})}: f_\theta =g\} \sim C_\t ||g||_{L^{p, q}_{\s}}.$$        This proves the estimate (\ref{embed1}).

We next prove the estimate (\ref{embed2}). Let $g\in [A^{p_0, q_0}_{\s_0}, A^{p_1, q_1}_{\s_1}]_\theta.$ We first suppose that $||g||^{anal}_\theta =0,$ i.e. $g=0$ in the Banach space $\left [A^{p_0, q_0}_{\s_0}, A^{p_1, q_1}_{\s_1}\right ]_\theta.$ We notice that
\begin{equation}\label{notice}
||\varphi||_{\mathcal F(A^{p_0, q_0}_{\s_0}, \hskip 1truemm A^{p_1, q_1}_{\s_1})}=||\varphi||_{\mathcal F\left (L^{p_0, q_0}_{\s_0}, \hskip 1truemm L^{p_1, q_1}_{\s_1}\right )}
\end{equation}
for all $\varphi \in \mathcal F\left (A^{p_0, q_0}_{\s_0}, A^{p_1, q_1}_{\s_1}\right ).$ This implies that 
$$||g||_{\left [L^{p_0, q_0}_{\s_0}, \hskip 1truemm L^{p_1, q_1}_{\s_1}\right ]_\theta}\leq ||g||_{\left [A^{p_0, q_0}_{\s_0}, \hskip 1truemm A^{p_1, q_1}_{\s_1}\right ]_\theta}$$
and hence $||g||_{\left [L^{p_0, q_0}_{\s_0}, L^{p_1, q_1}_{\s_1}\right ]_\theta} =0.$ By the estimate (\ref{lebesgue}), we obtain $||g||_{L^{p, q}_{\s}}=0.$

We next suppose that $0< ||g||^{anal}_\theta < \infty.$ There exists $\varphi \in \mathcal F(L^{p_0, q_0}_{\s_0}, L^{p_1, q_1}_{\s_1})$ such that $g=f_\theta$ and $||\varphi||_{\mathcal F\left (A^{p_0, q_0}_{\s_0}, A^{p_1, q_1}_{\s_1}\right )}\leq 2||g||^{anal}_\theta.$ By (\ref{lebesgue}) and (\ref{notice}), we obtain:
$$||g||_{A^{p, q}_{\s}}=||g||_{L^{p, q}_{\s}}\lesssim |g||_{\left [L^{p_0, q_0}_{\s_0}, L^{p_1, q_1}_{\s_1}\right ]_\theta} \lesssim ||\varphi||_{\mathcal F\left (A^{p_0, q_0}_{\s_0}, A^{p_1, q_1}_{\s_1}\right )} \leq 2||g||^{anal}_\theta.$$
This proves the estimate (\ref{embed2}).

%Then $||g||^{anal}_\theta < \infty$ and there exists $\varphi \in \mathcal F(A^{p_0, q_0}_{\s_0}, A^{p_1, q_1}_{\s_1})$ such that $g=\varphi_\theta$ and 
%$$||\varphi||_{\mathcal F\left (A^{p_0, q_0}_{\s_0}, A^{p_1, q_1}_{\s_1}\right)} < ||g||^{anal}_\theta +1.$$
%We check easily that $||\varphi||_{\mathcal F\left (A^{p_0, q_0}_{\s_0}, A^{p_1, q_1}_{\s_1}\right )}=||\varphi||_{\mathcal F\left (L^{p_0, q_0}_{\s_0}, L^{p_1, q_1}_{\s_1}\right)}.$ 
%the continuous embedding $\mathcal F(A^{p_0, q_0}_{\s_0}, A^{p_1, q_1}_{\s_1})\hookrightarrow \mathcal F(L^{p_0, q_0}_{\s_0}, L^{p_1, q_1}_{\s_1}).$ 
%This implies that
%$$||g||_{L^{p, q}_{\s}} \lesssim ||g||_{[L^{p_0, q_0}_{\s_0}, L^{p_1, q_1}_{\s_1}]_\theta} \lesssim ||g||^{anal}_\theta +1.$$
%In particular, we obtain the continuous embedding $[A^{p_0, q_0}_{\s_0}, A^{p_1, q_1}_{\s_1}]_\theta \hookrightarrow A^{p, q}_{\s}.$
\vskip 2truemm
(2) In this assertion, we have $\s_1=\s_2=\s.$ The weighted Bergman projector $P_{\s}$ extends to a bounded operator from $L^{p_i, q_i}_{\s}$ onto $A^{p_i, q_i}_{\s}, \hskip 2truemm i=0, 1$ and hence from $L^{p, q}_{\s}$ onto $A^{p, q}_{\s}.$ Then by Corollary 4.4, for every $g\in A^{p, q}_{\s},$ we have $P_\s g = g.$ The proof of assertion (2) is the same as the proof of assertion (1) with $\t=\s$ in the present case. More precisely, for the proof of the estimate (\ref{embed1}), we replace the mapping
$$
P_{\t}\circ f: \hskip 2truemm \overline {S} \rightarrow A^{p_0, q_0}_{\s_0}+ A^{p_1, q_1}_{\s_1}
$$  
with $f\in \mathcal F(L^{p_0, q_0}_{\s_0}, L^{p_1, q_1}_{\s_1}),$ by the mapping
$$
P_{\s}\circ f: \hskip 2truemm \overline {S} \rightarrow A^{p_0, q_0}_{\s}+ A^{p_1, q_1}_{\s}
$$  
with $f\in \mathcal F(L^{p_0, q_0}_{\s}, L^{p_1, q_1}_{\s}).$ The proof of the estimate (\ref{embed2}) remains the same.
%Here Theorem 4.2 is replaced by Theorems 4.3 and 4.4 respectively.
\vskip 2truemm
(3) We are going to prove the following more precise statement.

\begin{thm}
Let $\s \in \mathbb R^r$ be such that $s_k> \frac nr -1, \hskip 2truemm k=1,\cdots, r.$  Let $1\leq p_0, p_1\leq\infty$ and let $q_0, q_1$ be such that $1\leq q_0 < q_{\s}\leq q_1 <\infty.$ Assume that $P_{\s}$ extends to a bounded operator on $L^{p_1, q_1}_{\s}.$
Let $\theta, \varphi \in (0, 1)$ be related by the equation 
$$(\star) \quad \quad \frac 1{2} = \frac {1-\theta}{q_0}+\theta (\frac {1-\varphi}{2}+\frac {\varphi} {q_1})$$
%and define the exponent $1\leq p_2 \leq \infty$ by
%$\frac 1{p_2} = \frac {1-\theta}{p_0}+ \theta (\frac {1-\varphi}{p_2}+\frac \varphi {p_1}).$
%Furthermore we define the exponents $p_3, q_3$  by the system
%$$(\star \star) \quad 
%\left \{
%\begin{array}{clcr}
%\frac 1{p_2} &= \frac {1-\theta}{p_0}+\frac \theta {p_3}\\
%\frac 1{2} &= \frac {1-\theta}{q_0}+\frac \theta {q_3}\\
%\frac {{\s}_1}{2} &= \frac {(1-\theta){\s}_0}{q_0}+\frac {\theta {\s}_3} {q_3}
%\end{array}
%\right
%.
%$$
and assume that
$$(\star \star) \quad \quad\varphi <\frac {\frac 12 - \frac 1{q_{\s}}}{\frac 12 - \frac 1{q_1}}.$$
%Let $0<\varphi<\frac {\frac 12 -\frac 1{q_{\s}}}
%{\frac 12 - \frac 1{q_{\s} (p_1)}}$ and let $0<\theta<1$ related to $\varphi$ by the equation
%$$\frac 1{2} = \frac {1-\theta}{q_0}+\theta (\frac {1-\varphi}{2}+\frac {\varphi} {q_1}).$$
Then for $\xi = \frac {\theta \varphi}{1-\theta +\theta \varphi}, \hskip 2truemm \psi = \frac {\varphi}{1-\theta +\theta \varphi},$ we have
$$[A_{\s}^{p_0,q_0}, A_{\s}^{p_1,q_1}]_\xi=A_{{\s}_1}^{p_2, 2} \quad {\rm {and}} \quad [A_{\s}^{p_0,q_0}, A_{\s}^{p_1,q_1}]_\psi=A_{\s}^{p_3, q_3}$$ 
with equivalent norms, with
$$(\star \star \star) \quad \quad \frac 1{p_2}=\frac {1-\xi}{p_0}+\frac {\xi}{p_1}$$
and 
$$(\star \star \star \star) \quad 
\left \{
\begin{array}{clcr}
\frac 1{p_3} &= \frac {1-\varphi}{p_2}+\frac \varphi {p_1}\\
\frac 1{q_3} &= \frac {1-\varphi}{2}+\frac {\varphi} {q_1}
\end{array}
\right
.
.
$$  
\end{thm}

\begin{proof}
We apply the Wolff reiteration theorem (Theorem \ref{W}) with $A_1 = A^{p_0, q_0}_\s, \hskip 2truemm A_2 = A^{p_2, 2}_\s, \hskip 2truemm A_3 = A^{p_3, q_3}_\s$ and $A_4 = A^{p_4, q_4}_\s.$ On the one hand, we observe that $q_\s >2$ and hence the couple $(p_2, 2)$ satisfies the condition 
$$\frac 1{q_{\s} (p_2)} <\frac 12 < 1 - \frac 1{{q_{\s} (p_2)}}$$
of Theorem 1.1. So $P_s$ extends to a bounded operator on $L^{p_2, 2}_{\s}$ as well as we assumed that $P_{\s}$ extends to a bounded operator on $L^{p_1, q_1}_{\s}.$ We next apply assertion (2) of Theorem B to get the identity $[A_2, A_4]_\varphi = A_3$ with $p_3$ and $q_3$ defined by the system $(\star \star \star \star).$\\
On the other hand, the condition $(\star \star)$ and the definition of $q_3$ given by the second equality of $(\star \star \star \star)$ imply that $1<q_3<q_{\s}.$ We recall that $1<q_0<q_{\s}.$ Then by assertion (1) of Theorem B, we obtain the identity $[A_1, A_3]_\theta = A_2$ with
$$ 
\left \{
\begin{array}{clcr}
\frac 1{p_2} &= \frac {1-\theta}{p_0}+\frac \theta {p_3}\\
\frac 1{2} &= \frac {1-\theta}{q_0}+\frac {\theta} {q_3}
\end{array}
\right
.
$$  
The latter identity and the second identity of $(\star \star \star \star)$ give the relation $(\star).$ The former identity and the first identity of $(\star \star \star \star)$ give the relation $(\star \star \star).$

\end{proof}

\vskip 2truemm
{\bf Question.} Can Theorem 4.6 and consequently Theorem B be extended to other values of the interpolation parameters $p_0, q_0, \s_0, p_1, q_1, \s_1$?

\vskip 2truemm

{\bf Final remark.} We recall that $g\in [A_{{\s}_0}^{p_0,q_0}, A_{{\s}_1}^{p_1,q_1}]_\theta$ if there exists a mapping $f\in \mathcal F (A_{{\s}_0}^{p_0,q_0}, A_{{\s}_1}^{p_1,q_1})$ such that $f_\theta = g.$ For ${\s}_0 = {\s}_1$ real and $p_i = q_i, \hskip 2truemm i=1,...,r,$ an explicit construction was presented in \cite{BGN1} for such a mapping $f$ in terms of an analytic family of operators and the atomic decomposition of the relevant (usual) Bergman spaces and this construction was generalized in \cite{G} to mixed norm Bergman spaces associated to the same scalar parameter $\s=(\nu,\cdots,\nu)$. It may be interesting to extend this construction to  mixed norm Bergman spaces $A_{{\s}_0}^{p_0,q_0}, A_{{\s}_1}^{p_1,q_1}$ associated to more general vectors ${\s}_0, {\s}_1 \in \mathbb R^r.$ 
%%%%%%%%%%%%%%%%%%%%%%%%%%%%%%%%%%%%%%%%%%%%%%%%%%%%%%%
%%%%%%%%%%%%%%%%%%%%%%%%%%%%%%%%%%%%%%%%%%%%%%%%%%%%%%%%%%%%%%%%%%%

\bibliographystyle{plain}

\begin{thebibliography}{1}

%\bibitem{BB}
%\textsc{D. B\'ekoll\'e, A. Bonami}, \newblock{Analysis on tube
%domains over light cones: some extensions of recent results},
%\newblock{\em Actes des Rencontres d'Analyse Complexe},
% \newblock{Ed. Atlantique Poitiers} (2000), 17-37.

%\bibitem{BBG}
%\textsc{D. B\'ekoll\'e, A. Bonami, G. Garrig\'os},
% \newblock{Littlewood-Paley decompositions related to symmetric cones}, \newblock{\em
% IMHOTEP J. Afr. Math. Pures Appl}. \textbf{3} (2000), 11-41.

\bibitem{BBGR}
\textsc{D. B\'ekoll\'e, A. Bonami, G. Garrig\'os, F. Ricci},
\newblock{Littlewood-Paley decompositions related to symmetric cones and
Bergman projections in tube domains}, \newblock{Proc. London Math.
Soc}. (3) \textbf{89} (2004), 317-360.

%\bibitem{BBPR}
%\textsc{D. B\'ekoll\'e, A. Bonami, Peloso M. M., F. Ricci},
%\newblock{Boundedness of weighted Bergman projections on tube domains
%over light cones}, \newblock{\em Math. Zeit}. \textbf{237}
%(2001), 31-59.

\bibitem{BBGNPR}
\textsc{D. B\'ekoll\'e, A. Bonami, G. Garrig\'os, C. Nana, M. M. Peloso, F. Ricci}, 
\textit{Lecture notes on Bergman projectors in tube domains over cones: an analytic and geometric viewpoint},
IMHOTEP \textbf{5} (2004), Expos\'e I, Proceedings of the International Workshop in Classical Analysis, Yaound\'e
2001.
%\newblock{Lecture Notes of the Workshop "Classical Analysis,
%Partial Differential Equations and Applications ", Yaound\'e},
%December 10-15, (2001).

\bibitem{BT}
\textsc{D. B\'ekoll\'e,  A. Temgoua Kagou}, \newblock{Reproducing
properties and $L^p$-estimates for Bergman projections in Siegel
domains of type $II$}, {\em Studia Math}. \textbf{115} (3)
(1995), 219-239.

%\bibitem{BT1}
%\bysame, \newblock{Molecular decompositions and interpolation},
%\newblock{\em Int. Eq. Oper. Theory}.  \textbf{31} (1998),
%150-177.

\bibitem{BGN}
\textsc{D. B\'ekoll\'e, J. Gonessa and C. Nana},
\newblock{Lebesgue mixed norm estimates for Bergman projectors: from tube domains over homogeneous cones to homogeneous Siegel domains of type II},
\newblock{\em preprint}. 
%\textbf{Ser. I 337} (2003), 13-18.

\bibitem{BGN1}
\bysame,
\newblock{Complex interpolation between two weighted Bergman spaces on tubes over symmetric
cones}.
\newblock{\em C. R. Acad. Sci. Paris}. \textbf{Ser. I 337} (2003), 13-18.

\bibitem{BL}
\textsc{J. Bergh and J. L\"ofstr\"om,}
\newblock{\em Interpolation Spaces An Introduction},
\newblock{Springer-Verlag Berlin Heidelberg New York} (1976).

\bibitem{BN}
\textsc{A. Bonami and C. Nana,}
\newblock{Some questions related to the Bergman projection in symmetric domains},
\newblock{\em Adv. Pure Appl. Math.} {\textbf 6}, No. 4, 191-197 (2015).

%\bibitem{BP}
%\textsc{A. Benedek and R. Panzone},
%\newblock{The spaces $L^p$, with mixed norm},
%\ewblock{\em Duke Math. J}. \textbf{28} (1961), 301-324.

%\bibitem{BS}
%\textsc{C. Bennett and R. Sharpley},
%\newblock{\em Interpolation of Operators},
%\newblock{Pure and Applied Math.}, Vol. 129,
%\newblock{Academic Press, Inc.} (1988).

%\bibitem{Brezis}
%\textsc{H. Brezis},
%\newblock{\em Analyse fonctionnelle, Th\'eorie et applications}.
%\newblock{Masson, Paris} (1983).

%\bibitem{A}
%\textsc{A. Bonami}, \newblock{Three related problems on Bergman
%spaces of tube domains over symmetric cones}, \newblock{\em Rend.
%Mat. Acc. Lincei}, \textbf{13} (2002), 183-197.

\bibitem{BD}
\textsc{J. Bourgain and C. Demeter}, \newblock{The proof of the $l^2$-decoupling conjecture}, Ann. of Math. \textbf{182}, No. 1  (2015), 351-389.


\bibitem{C}
\textsc{M. Cwikel}, \newblock{Personal communication}.


\bibitem{CR}
\textsc{R. Coifman, R. Rochberg}, \newblock{Representation
theorems for holomorphic and harmonic functions in $L^p$},
\newblock{\em Ast\'erisque} \textbf{77} (1980), 11-66.

\bibitem{DD}
\textsc{D. Debertol},
\newblock{Besov spaces and boundedness of weighted Bergman projections over symmetric tube
domains},
\newblock{Dottorato di Ricerca in Matematica}, \newblock{\em Universit\`a
di Genova, Politecnico di Torino}, Aprile 2003.

\bibitem{DD1}
\bysame,
\newblock{Besov spaces and the boundedness of weighted Bergman projections over symmetric tube
domains},
\newblock{Publ. Mat.}, \textbf{49} (2005), 21-72.

\bibitem{FK}
\textsc{J. Faraut, A. Kor\'anyi}, {\em Analysis on symmetric
cones},
\newblock{Clarendon Press, Oxford} (1994).

\bibitem{G}
\textsc{J. Gonessa}, \newblock{Espaces de type Bergman dans les domaines homog\`enes de Siegel de type II: décomposition atomique et interpolation}.
\newblock{Th\`ese de Doctorat/Ph.D, Universit\'e de Yaound\'e I}  (2006).

\bibitem{JNP}
\textsc{S. Janson, P. Nilsson, J. Peetre (with Appendix by M. Zafran)},
\newblock{\em Notes on Wolff's note on interpolation spaces},
\newblock{Proc. London Math. Soc.}   \textbf{48} (3) (1984), 283-299.

\bibitem{Nana}
\textsc{C. Nana},
\textit{$L^{p,q}$-Boundedness of Bergman Projections
in Homogeneous Siegel Domains of Type II},
J. Fourier Anal. Appl.  \textbf{19} (2013), 997-1019.

\bibitem{NT}
\textsc{C. Nana and B. Trojan},
\newblock{}
\newblock{\em Ann. Scuola Norm. Sup. Pisa} Cl. Sci. (5), Vol. \textbf{X} (2011), 477-511.

%\bibitem{WR}
%\textsc{W. Rudin}, \newblock{\em Analyse r\'eelle et complexe},
%\newblock{Masson \'Ed. Paris} (1975).

\bibitem{RT}
\textsc{F. Ricci and M. Taibleson},
\newblock{Boundary Values of Harmonic Functions in Mixed Norm Spaces
  and Their Atomic Structure}.
\newblock{\em Ann. Scuola Norm. Sup. Pisa}, (4) \textbf{10} (1983), 1-54.

%\bibitem{S}
%\textsc{B. Sehba},
%\newblock{Bergman-type operators in tubular domains over symmetric cones}.
%\newblock{\em Proc. Edinburgh Math. Soc.} \textbf{52} (2009), 529-544.

\bibitem{SW}
\textsc{E. M. Stein,  G. Weiss}, {\em Introduction to Fourier
Analysis on Euclidean Spaces}, \newblock{Princeton University
Press, Princeton} (1971).

\bibitem{SW1} 
\textsc{E. M. Stein,  G. Weiss}, {\em Interpolation of operators with change of measures}, \newblock{Trans. Amer. Math. Soc.} \textbf{87}   (1958), 159-172.

\bibitem{W}
\textsc{T. Wolff},
\newblock{A note on interpolation spaces},
\newblock{\em Harmonic analysis} (1981 proceedings), \newblock{\em Lecture Notes in Math}.
\textbf{908}, \newblock{Springer Verlag} (1982), 199-204.

\bibitem{Z}
\textsc{K. Zhu},
\newblock{\em Operator Theory in function spaces},
\newblock{Marcel Dekker, New York} (1990).

\end{thebibliography}

\end{document}